\documentclass[a4paper,10pt]{article}
\usepackage{graphicx}
\usepackage{amssymb}
\usepackage{amsmath}
\usepackage{amsthm}

\newtheorem{definition}{Definition}[section]
\newtheorem{theorem}[definition]{Theorem}
\newtheorem{lemma}[definition]{Lemma}
\newtheorem{remark}[definition]{Remark}

\newtheorem{example}[definition]{\textit{Example}}
\newtheorem{proposition}[definition]{Proposition}

\usepackage{multicol}
\usepackage{enumerate}
\usepackage{graphicx}

\numberwithin{equation}{section}

\newcommand{\Bo}{\mathcal{B}}
\newcommand{\C}{\mathbb{C}}

\newcommand{\R}{\mathbb{R}}
\newcommand{\ide}{\operatorname{I}}
\newcommand{\Semi}{$C_0$-semigroup }
\newcommand{\DA}{D(A)}

\newcommand{\DAz}{D(A_{2})}
\newcommand{\ran}{\operatorname{ran}}
\newcommand{\blockA}{\left(\begin{array}{c}A_{1}\\A_{2} \quad 0\end{array}\right)}

\title{Generators with a closure relation}

\author{F.L. Schwenninger\thanks{The first author has been supported by the Netherlands Organisation 
for Scientific Research (NWO), grant no.\ 613.001.004.  \newline 
f.l.schwenninger@utwente.nl        }
        \qquad H. Zwart 
\\\small Dept. of Applied Mathematics, \\\small University of Twente, P.O. Box 217, \\\small 7500 AE Enschede, The Netherlands \\
}
                      
\begin{document}

\date{}
\maketitle
\subsection*{Abstract}
	Assume that a  block operator of the form $\left(\begin{smallmatrix}A_{1}\\A_{2}\quad 0\end{smallmatrix}\right)$, 
	acting on the Banach space $X_{1}\times X_{2}$, generates a contraction $C_{0}$-semigroup. 
	We show that the operator $A_{S}$ defined by $A_{S}x=A_{1}\left(\begin{smallmatrix}x\\SA_{2}x\end{smallmatrix}\right)$ 
	with the natural domain generates a contraction semigroup on $X_{1}$. 
	Here, $S$ is a boundedly invertible operator for which $\epsilon\ide-S^{-1}$ is dissipative for some $\epsilon>0$. 
	With this result the existence and uniqueness of solutions of the heat equation can be derived from the wave equation.
	\\ \textbf{Keywords}: {Block operator, Semigroup of operators, Semi-Inner-Product, Dissipative operator}
	\\ MSC Class.: {47D06 \and 47B44  \and 34G10}

\section{Introduction}\label{intro}
	The question whether an (unbounded) operator is the generator a \Semi appears naturally for abstract differential equations in the discussion of well-posedness. In this paper we relate the well-posedness of two abstract differential equations.\\
	Starting with an  abstract Cauchy problem (ACP) on the space $X_{1}\times X_{2}$,
		\begin{equation}\label{ACP-1}
		\begin{pmatrix}\dot{x}_{1}\\\dot{x}_{2}\end{pmatrix}=A_{ext}\begin{pmatrix}x_{1}\\x_{2}\end{pmatrix},\qquad x(0)=x_{0},\tag{ACP-1}
		\end{equation} 
	for an operator $A_{ext}$ of the form
		\begin{equation}
		A_{ext}=\blockA,\label{Aextform} \quad \begin{array}{l} A_{1}:D(A_{1})\subset X_{1}\times X_{2}\rightarrow X_{1}, \\
								    	   A_{2}:D(A_{2})\subset X_{1}\rightarrow X_{2},\end{array}
		\end{equation}
	we set $A_{S}x_{1}=A_{1}\left(\begin{smallmatrix}x_{1}\\SA_{2}x_{1}\end{smallmatrix}\right)$ 
	where $S$ is a bounded operator, and define the ACP
	\begin{equation}\label{ACP-2}
		\dot{x}=A_{S}x,\qquad x(0)=x_0\in X_{1}.\tag{ACP-2}
		\end{equation} 
	The question is whether (\ref{ACP-2}) is well-posed when (\ref{ACP-1}) is assumed to be well-posed.\\
The idea comes from port-based modeling, see e.g.\ \cite{jacobzwartbook,villegasthesis}. There, $A_{ext}$ defines a structure relating the variables $(f_{1},f_{2})^{T}$ and $(e_{1},e_{2})^{T}$, by $f=A_{ext}e$. 
	Now, adding the closure relation $e_{2}=Sf_{2}$, where $S$ maps from $X_{2}$ to $X_{2}$, yields the structure $A_{S}$, as depicted in Figure \ref{intstruct}. There, the operator $S$ is seen as adding dissipation. \newline
	 
	The form (\ref{Aextform}) appears in the context of port-Hamiltonian systems, see \cite{gorrecmaschkevillegaszwart,villegasthesis}, but is applicable in wider settings, see \cite{ZwartParahyp}. Motivated by this, we will study well-posedness in terms of operators generating contraction semigroups.
	Hence, we want to know whether the operator $A_{S}$ will generate a contraction \Semi 
	if this holds for the initial system of $A_{ext}$. 
	The case of $X_{1}$ and $X_{2}$ being Hilbert spaces has already been solved 
	and can be found in \cite{gorrecmaschkevillegaszwart,villegasthesis,ZwartParahyp}. 
	Our aim is to generalize the result, including the conditions on $S$, to arbitrary Banach spaces.\newline 
	A natural application is given by the heat equation for the space $L^{1}$. We conclude existence and uniqueness of its solutions from the undamped wave equation. Motivated by the example we give further results concerning the analyticity of the semigroup generated by $A_{S}$.
		\begin{figure}
		\begin{center}
		\includegraphics[scale=0.7]{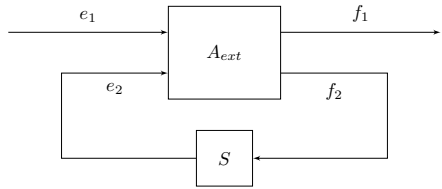}
		\caption{Interconnection structure}
		\end{center}
		\label{intstruct}
		\end{figure}
\subsection{Semi-Inner-Products}\label{lemmata}
		In this section we collect some facts we are going to need.\\
		The following notion was introduced by Lumer in 1961, see \cite{LumerSIP}. From now on, $X$ will be a Banach space.
		\begin{definition} For a Banach space $X$, a mapping $[\cdot,\cdot]:X\times X\to \C$ is called \textit{semi-inner-product, SIP}, if  for all $x,y,z\in X$ and $\lambda\in\C$

\begin{itemize}
				  \item $[x+\lambda z,y]=[x,y]+\lambda[z,y]$\qquad\textit{(linearity in first component)}, 
				  \item $[x,x]=\|x\|^{2}$\qquad \textit{(positive definiteness)},
				  \item $|[x,y]|^{2}\leq[x,x][y,y]$\qquad  \textit{(Cauchy-Schwarz inequality)}.	
				  \end{itemize}	
		\end{definition}
		
				\begin{lemma}\label{lemmaSIP} 
			The following assertions hold
			\begin{enumerate}[i.]
			\item \label{leSIP1} Every Banach space $X$ has a SIP, i.e. $X$ is a \textit{SIP space}.
			\item\label{leSIP2} For SIP spaces $(X,[\cdot,\cdot]_{X})$, $(Y,[\cdot,\cdot]_{Y})$, the mapping defined by
			\begin{equation}\label{prodSIP}
				\Big[\begin{pmatrix}x_{1}\\y_{1}\end{pmatrix},\begin{pmatrix}x_{2}\\y_{2}\end{pmatrix}\Big]_{X\times Y}:=[x_{1},x_{2}]_{X}+[y_{1},y_{2}]_{Y}
			\end{equation} is a SIP for $X\times Y$ equipped with the Euclidean norm
			\begin{equation}\label{twonorm}
			\left\|\begin{pmatrix}x\\y\end{pmatrix}\right\|_{X\times Y}=\sqrt{\|x\|^{2}+\|y\|^{2}},\qquad x\in X,y\in Y.
			\end{equation} 
			\end{enumerate}
		\end{lemma}
			\begin{proof}
			 \ref{leSIP1}.\  relies on the Hahn-Banach theorem, see \cite{LumerSIP}. 
			For \ref{leSIP2}.\ one simply checks the definition of a SIP.
			\end{proof}
		As an example, let us consider $L^{p}$ spaces, see \cite[page 90]{EngelNagel}.
	 	\begin{example}\label{SIPLp}
			 For the space $L^{p}[0,1]$, $p\geq1$, 
			 \begin{equation*}
			 [f,g]= \int_{0}^{1}f(s)\tilde{g}(s) \ ds,\qquad f,g\in L^{p}[0,1],
			 \end{equation*}
			  where 
			\begin{equation*}
			  \tilde{g}(s):=\left\{ \begin{array}{cl} \overline{g(s)}|g(s)|^{p-2}\|g\|_{L^{p}}^{2-p}, &g(s)\neq0\\
                                0  &\text{otherwise} \end{array}\right. ,
			 \end{equation*}
			  defines a SIP.
		\end{example}
		\begin{definition} Let $X$ be a Banach space. An operator $A:\DA\subset X\rightarrow X$ is called \textit{dissipative}, if there exists a SIP, such that
			\begin{equation}\label{defDiss}
			\Re[Ax,x]\leq0 \qquad \forall x\in \DA.
			\end{equation}
		\end{definition}
		
In the literature the notion of \textit{dissipativity} for general Banach spaces is often introduced in a different way (see e.g.\ \cite{EngelNagel}). 
				We remark that this definition is equivalent. 
				For instance, (\ref{defDiss}) implies that for all $\lambda>0$, $x\in X$,
				\begin{equation*}
					\lambda\|x\|^{2}=\lambda\Re[x,x]=\Re[(\lambda\ide -A)x,x]+\Re[Ax,x]\leq\|(\lambda\ide-A)x\| \cdot \|x\|,
				\end{equation*}
				where we used (\ref{defDiss}) and Cauchy-Schwarz in the last inequality. 
				The converse employs the Banach-Alaoglu Theorem and can be found in Proposition II.3.23 in \cite{EngelNagel}. There,  (\ref{defDiss}) is formulated as 
				\begin{align*}
				\forall x\in\DA\  \exists j(x)\in \mathcal{J}(x)&:= \left\{x'\in X':\langle x,x'\rangle=\|x'\|^{2}=\|x\|^{2}\right\} \text{ such that } \\& \Re\langle Ax,j(x)\rangle\leq 0,
				\end{align*}
				(where $X'$ denotes the dual of $X$, $\langle \cdot,\cdot\rangle$ the duality brackets). $\mathcal{J}(x)$ is called the \textit{duality set of $x$}. Note that any \textit{selection} $j:X\rightarrow X':x\mapsto j(x)\in\mathcal{J}(x)$ defines a SIP $[\cdot,\cdot]=\langle \cdot,j(\cdot)\rangle$ and, vice versa, every SIP $[\cdot,\cdot]$ yields a selection $j(x)=[\cdot,x]\in\mathcal{J}(x)$ for all $x\in X$. \\
				
			The following theorem is a standard result in semigroup theory and can be found in \cite[Section II.3.b]{EngelNagel} or \cite[Theorem 3.1]{LumPhiDiss} (in the latter dissipativity is defined via SIPs). 

		\begin{theorem}[Lumer-Phillips]\label{LumerPhillips} 
			For the linear operator $A$ on the Banach space $X$ the following assertions are equivalent
				\begin{enumerate}[i.]
					\item \label{LumPhil1} $A$ generates a contraction $C_{0}$-semigroup,
					\item \label{LumPhil2} $A$ is densely defined, dissipative and there exists some $\lambda>0$ such that 
						\begin{equation}\label{rancond}
							\ran(\lambda\ide-A)=X.
						\end{equation}
				\end{enumerate}
			In this case $A$ is dissipative w.r.t.\ any SIP on $X$, and (\ref{rancond}) holds for every $\lambda>0$.
			If $X$ is reflexive, $D(A)$ is automatically dense from the other assumptions in \ref{LumPhil2}. 
		\end{theorem}

\section{Main result}\label{mainsection}
	\begin{theorem}\label{mainresult}
    		   Let $A_{1}:D(A_{1})\subset X_{1}\times X_{2}\rightarrow X_{1}$ and $A_{2}:\DAz\subset X_{1}\rightarrow X_{2}$  be operators such that 
	            \begin{align*}
	            	 & A_{ext}:=\blockA,\\
	           	 &D(A_{ext})=\left\{(x_{1},x_{2})\in X_{1}\times X_{2}:x_{1}\in D(A_{2})\wedge(x_{1},x_{2})\in D(A_{1})\right\}
	            \end{align*}
	          generates a contraction $C_{0}$-semigroup on $X_{1}\times X_{2}$ equipped with the Euclidean norm, see (\ref{twonorm}).
	          Let $S\in\Bo(X_{2})$ be a boundedly invertible satisfying 
	            \begin{equation}\label{assSinvers}
	            \Re[x,Sx]_{2} \geq m_{2}\|x\|_{2}^2 \qquad \forall x\in X_{2},
	            \end{equation}
	           for some $m_{2}>0$ and some SIP $[\cdot,.\cdot]_{2}$ on $X_{2}$. Then
	            \begin{equation*}
	            A_{S}x=A_{1}\begin{pmatrix}x\\SA_{2}x\end{pmatrix},
	            \end{equation*}
	           defined on $D(A_S)=\left\{x\in X_{1}:(x,SA_{2}x)\in D(A_{ext})\right\}$         
	           generates a contraction semigroup on $X_1$ provided that $D(A_{S})$ is dense or that $X_{1}$ is reflexive.
	 \end{theorem}
			\begin{proof}
				By the Lumer-Phillips Theorem, the proof consists of two steps. 
				First we show that $A_{S}$ is dissipative. 
				Let $[\cdot,\cdot]_{1}$ be a SIP on $X_{1}$. Then, let $[\cdot,\cdot]_{X_{1}\times X_{2}}$ be the SIP 
				defined in (\ref{prodSIP}) with respect to $[\cdot,\cdot]_{1}$ and $[\cdot,\cdot]_{2}$.
				 For $x\in D(A_{S})$ we get
			\begin{align}
				 [A_{S}x,x]_{1}={}&\begin{bmatrix}A_{1}\begin{pmatrix}x\\SA_{2}x\end{pmatrix},x\end{bmatrix}_{1}\notag\\
				                         			   ={}&\begin{bmatrix}A_{1}\begin{pmatrix}x\\SA_{2}x\end{pmatrix},x\end{bmatrix}_{1}+
				                         			     \begin{bmatrix}A_{2}x,SA_{2}x\end{bmatrix}_{2}-\begin{bmatrix}A_{2}x,SA_{2}x\end{bmatrix}_{2}\notag\\
				                         			   ={}&\begin{bmatrix}A_{ext}\begin{pmatrix}x\\SA_{2}x\end{pmatrix},\begin{pmatrix}x\\SA_{2}x\end{pmatrix}\end{bmatrix}_{X_{1}\times X_{2}}
				                         			       -\begin{bmatrix}A_{2}x,SA_{2}x\end{bmatrix}_{2}\label{Aextdiss}
			\end{align}
			The second term is less or equal zero by the assumption (\ref{assSinvers}).
			By Theorem \ref{LumerPhillips}, $A_{ext}$ is dissipative w.r.t. any SIP on $X_{1}\times X_{2}$. 
			Together this yields
				\begin{equation*}
					\Re[A_{S}x,x]_{1}\leq0.
				\end{equation*}
			Hence, $A_{S}$ is dissipative. \\ 
			
			To show the range condition (\ref{rancond}), let $\lambda\in\R$ and consider
				 \begin{equation*}
                  			  P=\begin{pmatrix}0  &0\\\\0 & \lambda\ide-S^{-1}\end{pmatrix}\in\Bo(X_{1}\times X_{2}).
                   		\end{equation*}
			$A_{ext}+P$ is a bounded perturbation of a generator, hence, it also generates a semigroup, 
			see \cite[Theorem III.1.3]{EngelNagel}.
			By (\ref{assSinvers}) we have for $x=(x_{1},x_{2})^{T}\in X_{1}\times X_{2}$ that
				\begin{equation*}
					\Re[Px,x]_{X_{1}\times X_{2}}=\Re[(\lambda\ide-S^{-1})x_{2},x_{2}]_{2}\leq\left(\lambda-\frac{m_{2}}{\|S\|^{2}}\right)\|x_{2}\|^{2}.
				\end{equation*}
			Thus, $P$ is dissipative if $\lambda\in (0,m_{2}/\|S\|^{2}]$, and then, 
			$A_{ext}+P$ generates a contraction semigroup by the Lumer-Phillips Theorem.
			Particularly, the range of $\lambda\ide-A_{ext}-P$ equals $X_{1}\times X_{2}$. 
			Hence, for any pair $(g,0)\in X_{1}\times X_{2}$ there exists $(x_{1},x_{2})\in X_{1}\times X_{2}$ such that
				\begin{equation}\label{proofmainthmRan}
					(\lambda\ide-A_{ext}-P)\begin{pmatrix}x_{1}\\x_{2}\end{pmatrix}=\begin{pmatrix}g\\0\end{pmatrix}.
				\end{equation}
			By the structure of $A_{ext}$, the second component reads
				\begin{equation*}
					\lambda x_{2}-A_{2}x_{1}+S^{-1}x_{2}-\lambda x_{2}=0,
				\end{equation*}
			which implies $x_{2}=SA_{2}x_{1}$. Inserting in the first component of (\ref{proofmainthmRan}) gives
				\begin{equation*}
				\lambda x_{1}-A_{1}\begin{pmatrix}x_{1}\\SA_{2}x_{1}\end{pmatrix}=g,
				\end{equation*}
			which is $(\lambda\ide-A_{S})x_{1}=g$. Thus, $\ran (\lambda\ide-A_{S})=X_{1}$.\\
			By assumption that either $D(A_{S})$ is dense or $X_{1}$ is reflexive we conclude from Theorem \ref{LumerPhillips} (Lumer-Phillips) that $A_{S}$ generates a contraction semigroup.
			\end{proof}
			\begin{remark}
			\begin{enumerate}
			\item Because of the boundedness of $S^{-1}$, condition (\ref{assSinvers}) holds for all SIPs on $X_{2}$ if it holds for some SIP, see \cite[Remark 2]{LumPhiDiss}. 
			\item Note that since $S$ is boundedly invertible, (\ref{assSinvers}) is equivalent to
			\begin{equation*}
			\exists \tilde{m}>0\ \forall x\in X_{2}:\qquad \Re[S^{-1}x,x]_{2}\leq \tilde{m} \|x\|^{2} \Leftrightarrow \Re[(\tilde{m}\ide-S^{-1}\ide)x,x]_{2}\leq0,
			\end{equation*}
			which means that $\tilde{m}\ide-S^{-1}$ is dissipative.
			\item For a boundedly invertible operator $B\in\Bo(X)$ on a Banach space $X$, $B$ dissipative does not necessarily imply that $B^{-1}$ is dissipative. In fact, by Lumer-Phillips this is equivalent to ask whether $B^{-1}$ generates a contraction $C_{0}$-semigroup, if $B$ does. The answer is negative in general, even in finite dimensions, see e.g. \cite[Section 2]{GomilkoTomilovZwart}. However, on Hilbert spaces, the dissipativity of $B^{-1}$ always  follows from the one of $B$ by the symmetry of the inner product.
			\item For $X_{2}$ being a Hilbert space the assumptions on $S$ are equivalent to
			\begin{equation*}
			S\in\Bo(X_{2}) \text{ and } S+S^{*}\geq\epsilon\ide>0.
			\end{equation*}	
			\end{enumerate}		
			\end{remark} 
			We finish this part by showing that the  converse of Theorem \ref{mainresult} does not hold in the sense that $A_{ext}$ does not necessarily generate a contraction $C_{0}$-semigroup if $A_{S}$ does.
			Looking at the proof, there is no reason to believe that the arguments in both parts (disspativity, range condition) could be reversed. For instance, let $S=\ide$ and $A_{S}$ be dissipative. Then, one gets that 
\begin{equation*}\Re\begin{bmatrix}A_{ext}\begin{pmatrix}x\\SA_{2}x\end{pmatrix},\begin{pmatrix}x\\SA_{2}x\end{pmatrix}\end{bmatrix}_{X_{1}\times X_{2}}\leq \|x\|_{X_{1}}^{2}\qquad \forall x\in D(A_{S})
\end{equation*}
by reading the eq.\ (\ref{Aextdiss}) in reversed order. However, this won't give that $A_{ext}$ is dissipative (and since $A_{ext}$ should generate a semigroup, this should hold w.r.t.\ any SIP) in general. In fact, consider the matrix case
\begin{equation*}
A_{ext}=\begin{pmatrix}0&0\\1&0\end{pmatrix}\in\R^{2\times2} \quad \Rightarrow \quad A_{S}=A_{\ide}=0,
\end{equation*}
with the Euclidean norm on $\R^{2}$. Clearly,
\begin{equation*}
\left[A_{ext}\begin{pmatrix}x_{1}\\x_{2}\end{pmatrix},\begin{pmatrix}x_{1}\\x_{2}\end{pmatrix}\right]=\left[\begin{pmatrix}0\\x_{1}\end{pmatrix},\begin{pmatrix}x_{1}\\x_{2}\end{pmatrix}\right]=x_{1}x_{2}.
\end{equation*}
Therefore, $A_{ext}$ can not be dissipative, whereas $[A_{S}x,x]=0$.

%
%
\subsection{From Wave to Heat equation}\label{WaveandHeat}

			We start with the undamped wave equation $\frac{\partial^{2}w}{\partial t^{2}}(\xi,t)=\frac{\partial^{2}w}{\partial \xi^{2}}(\xi,t)$ on $[0,1]$.
			The boundary conditions are chosen to be
			\begin{equation}\label{wave-BC}\left\{
			\begin{array}{l}
				(K_{1}-1)\frac{\partial w}{\partial t}(1,t)=(K_{1}+1)\frac{\partial w}{\partial \xi}(1,t),\\
				(1-K_{2})\frac{\partial w}{\partial t}(0,t)=(K_{2}+1)\frac{\partial w}{\partial \xi}(0,t),  
			\end{array}\right. \forall t\geq0, \text{with }|K_{1}|,|K_{2}|\leq 1.
 			\end{equation}
			This can be written as the following ACP on $L^{p}[0,1]\times L^{p}[0,1]$, $p\geq1$,
			\begin{equation}\label{ACPwave}
                       		\begin{pmatrix}\dot{x}_{1}\\\dot{x}_{2}\end{pmatrix}=\begin{pmatrix}0&\frac{\partial}{\partial \xi}\\\frac{\partial}{\partial \xi}&0\end{pmatrix}
			\begin{pmatrix}x_{1}\\x_{2}\end{pmatrix}:=A_{ext}x,	\qquad x(0)=x_{0}
			\end{equation}
	 		with
			\begin{align}\label{Aextwave}
           D(A_{ext})=\left\{\begin{pmatrix}f_{1}\\f_{2}\end{pmatrix}\in(L^{p}[0,1])^{2}:f_{1},f_{2}\text{ abs.\ continuous and }\right. \\\notag \left.\frac{\partial f_{1}}{\partial \xi},\frac{\partial f_{2}}{\partial \xi}\in L^{p}[0,1],
                			 (Qf)_{1}(1)=K_{1}(Qf)_{2}(1),(Qf)_{2}(0)=K_{2}(Qf)_{1}(0)\begin{matrix} \ \\ \ \end{matrix}\right\},
			\end{align}
			where $Q=\frac{1}{\sqrt{2}}\left(\begin{smallmatrix}1&1\\-1&1\end{smallmatrix}\right)$. 
			In the framework of Theorem \ref{mainresult} the operators $A_{1}$ and $A_{2}$ read
			\begin{align*}
			D(A_{1})=\left\{\begin{pmatrix}f_{1}\\f_{2}\end{pmatrix}\in(L^{p}[0,1])^{2}:f_{1},f_{2}\text{ abs.\ continuous and } \frac{\partial f_{2}}{\partial \xi}\in L^{p}[0,1],\right.\\
                			 \left. (Qf)_{1}(1)=K_{1}(Qf)_{2}(1),(Qf)_{2}(0)=K_{2}(Qf)_{1}(0)\begin{matrix} \ \\ \ \end{matrix}\right\}, \quad A_{1}\begin{pmatrix}f_{1}\\f_{2}\end{pmatrix}=\frac{\partial}{\partial\xi}f_{2},
			 \end{align*}
			 \begin{equation*}
			D(A_{2})=\left\{f\in L^{p}[0,1]:f\text{ abs.\ cont.}, \frac{\partial f}{\partial \xi}\in L^{p}[0,1]\right\}, \quad A_{2}f=\frac{\partial}{\partial\xi}f.
			\end{equation*}
			By diagonalizing, $\mathcal{D}=QA_{ext}Q^{-1}$, it is easy to show that $A_{ext}$ generates a contraction $C_0$-semigroup (in the Euclidean norm). 
%
%
			Furthermore, let $\xi\mapsto\lambda(\xi)$ be positive and continuously differentiable on $[0,1]$ 
			and denote by $S$ the induced multiplication operator. Then, 
			\begin{align}\label{defASheat}
			    A_{S}f={}&A_{1}\begin{pmatrix}f\\SA_{2}f\end{pmatrix}=\big(0\quad \frac{\partial}{\partial \xi}\big)\begin{pmatrix}f\\\lambda(\xi)\frac{\partial f}{\partial \xi}\end{pmatrix}=\frac{\partial}{\partial \xi}\left(\lambda(\xi)\frac{\partial f}{\partial \xi}\right),\\
			    			 D(A_{S})={}&\left\{f\in L^{p}[0,1]:(f,SA_{2}f)^{T}\in D(A_{ext})\right\}\notag
                		\end{align}
			By the assumptions on $\lambda(\xi)$, it follows easily that
			 \begin{align*}
			   D(A_{S})=\left\{f\in L^{p}[0,1]:f,\frac{\partial f}{\partial \xi}\text{ abs.\ continuous and }\frac{\partial f}{\partial \xi},\frac{\partial^{2} f}{\partial \xi^{2}}\in L^{p}[0,1],\right. \\
                			 \quad \left.(K_{1}+1)f(1)=(K_{1}-1)\lambda(1)\frac{\partial f}{\partial \xi}(1),(K_{2}+1)f(0)=(1-K_{2})\lambda(0)\frac{\partial f}{\partial \xi}(0) \right\}
		         \end{align*}
			which is dense in $L^{p}[0,1]$.
			The operator $A_S$ corresponds to the heat equation
			 \begin{equation}\label{heat}
				 \frac{\partial u}{\partial t}(\xi,t)=\frac{\partial}{\partial \xi}\left(\lambda(\xi)\frac{\partial u}{\partial \xi}(\xi,t)\right),
                          \end{equation}
  			with the Robin boundary conditions
			 \begin{equation}\label{heatBC}\left\{\begin{array}{l}
			 (K_{2}+1)u(0,t)=(1-K_{2})\lambda(0)\frac{\partial u}{\partial\xi}(0,t),\\
			  (K_{1}+1)u(1,t)=(K_{1}-1)\lambda(1)\frac{\partial u}{\partial\xi}(1,t),
			\end{array}\right. \forall t\geq0.
			 \end{equation}
			 Hence, $\lambda(\xi)$ can represent the heat conduction coefficient.
			 It remains to show that the assumptions on $S$ are fulfilled. 
			 Clearly, $S$ is a bounded operator which is boundedly invertible since there exists $\lambda_{min},\lambda_{max}$ such that $0<\lambda_{min}<\lambda(\xi)<\lambda_{max}$ for $\xi\in[0,1]$. To show (\ref{assSinvers}) we use the SIP from Example \ref{SIPLp},
			 \begin{equation}
			 [f,Sf]=\int_{0}^{1}\lambda(s)^{p-1}|f(s)|^{p}\|Sf\|_{L^{p}}^{2-p} \ ds\geq \frac{1}{\lambda_{max}}\|Sf\|_{L^{p}}^{2}\geq  \frac{\lambda_{min}^{2}}{\lambda_{max}}\|f\|_{L^{p}}^{2}.
			 \end{equation}
			 Thus, by Theorem \ref{mainresult}, we conclude that $A_{S}$ generates a contraction semigroup.

\subsection{Further results}
Motivated by the example in Subsection \ref{WaveandHeat}, one might ask when $A_{S}$ is even generating an analytic semigroup. Without further assumptions on the operator $A_{ext}$ this does not seem to work in general. 
However, the following theorem gives an answer.

\begin{theorem}\label{thm:group}
Assume that $A_{ext}$ from Theorem \ref{mainresult} has the form
\begin{align}\label{Aextspecial}
&A_{ext}=\left(\begin{array}{cc}0& A_{12}\\A_{21} & 0\end{array}\right),\\ \notag
&D(A_{ext})=\left\{(x_{1},x_{2})\in X_{1}\times X_{2}:x_{1}\in D(A_{21}),x_{2}\in D(A_{12})\right\}
\end{align}
and that $\mathcal{A}:=\left(\begin{smallmatrix}\ide&0\\0&S\end{smallmatrix}\right)A_{ext}$ generates a $C_{0}$-group, where $S\in\Bo(X_{2})$. Then, 
\begin{equation*}
A_{S}=A_{12}SA_{21},
\end{equation*}
 with $D(A_{S})=\left\{x\in X_{1}:x\in D(A_{21}),SA_{21}x\in D(A_{12})\right\}$ generates an analytic semigroup of angle $\frac{\pi}{2}$. 
\end{theorem}
\begin{proof}
It is a fact that if $\mathcal{A}$ generates a $C_{0}$-group, it follows that $\mathcal{A}^{2}$ generates an analytic $C_{0}$-semigroup of angle $\frac{\pi}{2}$, see \cite[Corollary II.4.9]{EngelNagel}. Therefore, the result follows by considering the upper left entry of 
\begin{equation*}
\mathcal{A}^{2}=\begin{pmatrix}0&A_{12}\\SA_{21}&0\end{pmatrix}\begin{pmatrix}0&A_{12}\\SA_{21}&0\end{pmatrix}=\begin{pmatrix}A_{12}SA_{21}&0\\0&A_{21}SA_{12}\end{pmatrix},
\end{equation*}
where 
\begin{equation*}
	D(\mathcal{A}^{2})=\left\{(x_{1},x_{2})\in D(A_{21})\times D(A_{12}):SA_{21}x_{1}\in D(A_{12}), SA_{12}x_{2}\in D(A_{21}\right\}.
	\end{equation*}
\end{proof}

\begin{remark}
Given that $A_{ext}$ generates a $C_{0}$-group, the assumption in Theorem \ref{thm:group}, that $\mathcal{A}$ generates a $C_{0}$-group, can be checked by means of (multiplicative) perturbation results for generators, see e.g. \cite{GustafsonLumer}.
\end{remark}
In the following we note that the group generation is not surprising in the view of the assumptions in Theorem \ref{mainresult}

\begin{proposition}[Lemma 5.1 in \cite{ZwartParahyp}]\label{propgroup}
Let $A_{ext}$, given in the form (\ref{Aextspecial}), generate a $C_{0}$-semigroup $T(t)$  with constants $M,\omega$ such that $\|T(t)\|\leq M e^{t\omega}$ for all $t>0$. Then, $A_{ext}$ can be extended to a $C_{0}$-group which satisfies $\|T(t)\|\leq Me^{|t|\omega}$. In particular, if $A_{ext}$ generates a contraction semigroup, then $A_{ext}$ generates a group of isometries.
\end{proposition}
With the results of this subsection we are able to continue the discussion of the example of the wave and heat equation in Section \ref{WaveandHeat}. To conclude the analyticity of the semigroup generated by $A_{S}$, (\ref{defASheat}), it remains to check that $\mathcal{A}=\left(\begin{smallmatrix}\ide&0\\0&S\end{smallmatrix}\right)A_{ext}$ generates a $C_{0}$-group. By Proposition \ref{propgroup}, it even suffices to show that $\mathcal{A}$ generates a $C_{0}$-semigroup.
In fact, by diagonalizing and using the specific assumptions on $S$ (the multiplication operator induced by $\lambda$), this is not hard to deduce (see also \cite[Chapters 12 and 13]{jacobzwartbook}). 

\subsection{Remarks and Outlook}
	One might question the use of SIPs instead of employing the more common dissipativity definition only relying on the norm. The reason is that the condition on $S$ and the proof happens to be natural in the view of the Hilbert space result.  \\
	Discussing more general $S$ (and at the same time restricting the form of $A_{ext}$) as $S=i\ide$, like it is done in \cite[Section 4]{ZwartParahyp} for Hilbert spaces, might be possible as well as adaptions to nonlinear $S$.

\end{document}